\documentclass[12pt]{amsart}
 \usepackage{mathtools}
\mathtoolsset{showonlyrefs,showmanualtags} 

\usepackage{amssymb, amsmath, amsthm, esint}
\usepackage{mathrsfs}
\usepackage{bbm,dsfont}
\usepackage{hyperref}
\usepackage{mathtools}
\usepackage{fullpage}
\usepackage{color}

\usepackage[backrefs]{amsrefs}

 \usepackage[top=1.5in, bottom=1.5in, left=1.25in, right=1.25in]	{geometry}
\usepackage[all]{xy}
\usepackage{amsmath}
\usepackage{amssymb}
\usepackage{amsthm}
\usepackage{amscd}

\numberwithin{equation}{section}

\newtheorem{theorem}[equation]{Theorem}
\newtheorem{definition}[equation]{Definition}
\newtheorem{proposition}[equation]{Proposition}

\newtheorem{lemma}[equation]{Lemma}

\usepackage{hyperref} 
\hypersetup{
    colorlinks=true,       
    linkcolor=blue,          
    citecolor=magenta,        
    filecolor=magenta,      
    urlcolor=cyan           
}

\begin{document}
\title{Multi-Frequency Oscillation Estimates Arising in Pointwise Ergodic Theory}

\author{Ben Krause}
\address{
Department of Mathematics,
University of Bristol \\
Woodland Rd, Bristol BS8 1UG}
\email{ben.krause@bristol.ac.uk}
\date{\today}

\maketitle

\begin{abstract}
We prove essentially optimal $L^p(\mathbb{R})$-estimates for variational variants of the maximal Fourier multiplier operators considered by Bourgain in his work on pointwise convergence of polynomial ergodic averages. As a corollary of our methods, we are able to quickly extend a result of Bourgain, namely the pointwise convergence of 
ergodic averages of integer parts of real-variables polynomials, to a broader class of functions, previously considered in a wide range of contexts by Boshernitzan-Jones-Wierdl. Namely, the following averages converge almost everywhere
\[ \frac{1}{N} \sum_{n \leq N} T^{\lfloor P(n) \rfloor} f, \; \; \; f \in L^p(X,\mu), \ P \in \mathbb{R}[\cdot], \]
for any $\sigma$-finite measure space equipped with a measure-preserving transformation, $T:X \to X$, whenever $1 < p  \leq \infty$ if $P$ is linear, and $4/3 < p \leq \infty$ otherwise.
\end{abstract}

\section{Introduction}
The topic of this paper is the oscillation of so-called ``multi-frequency" operators arising in pointwise ergodic theory. Here is the set-up:

Let $\Theta \subset \mathbb{R}$ be a finite set of frequencies, normalized so that 
\[ 2 \geq \min_{\theta \neq \theta' \in \Theta} |\theta - \theta'| > 1,\]
and for each $k$, consider a $2^{-k}$ neighborhood of $\Theta$,
\[ R_k := R_{k,\Theta} := \Theta + (-2^{-k},2^{-k}).\]
In his celebrated work on pointwise convergence of ergodic averages along polynomial orbits \cite{B2}, Bourgain established the following maximal estimate involving Fourier projections onto $\{ R_k\}$; we will use the notation
\[ \Xi_k f := \Xi_{k,\Lambda} f := \big( \mathbf{1}_{R_k} \hat{f} \big)^{\vee}
\]
and below set $N := |\Theta|$.

\begin{theorem}\label{t:max}
There exists an absolute $0 < C< \infty$ so that following estimate holds:
\[ \| \sup_k |\Xi_k f| \|_{L^2(\mathbb{R})} \leq C \log^2 N \cdot \|f \|_{L^2(\mathbb{R})}.
\]
\end{theorem}
Slightly simpler variants of this estimate (involving smooth cut-offs) have had a wide range of applications in questions involving pointwise convergence \cite{B00,B3,D1,L}, as well as some more subtle ones \cite{D1,DOP}. And, Theorem \ref{t:max} is known to be essentially sharp \cite{BKO}.

To describe our modifications, we recall two standard ways to measure oscillation, namely via \emph{jump-counting} and \emph{$r$-variation}, classically used in probability theory, and imported to the harmonic-analytic setting by Bourgain in \cite{B2}. We begin with the relevant definitions.

\begin{definition}
    For a sequence of scalars $\{ a_k \} \subset \mathbb{C}$ define the (greedy) jump-counting function at altitude $\lambda > 0$,
    \[ N_\lambda(a_k) := \sup\{ M : \text{There exists } k_0 < k_1 < \dots < k_M : |a_{k_n} - a_{k_{n-1}}| > \lambda \text{ for each } 1 \leq n \leq M \}.\]
And for each $0 < r < \infty$, define the \emph{$r$-variation} to be
\[ \mathcal{V}^r(a_k) := \sup \big( \sum_n |a_{k_n} - a_{k_{n-1}}|^r \big)^{1/r} \]
where the supremum runs over all finite increasing subsequences; we define
\[ \mathcal{V}^{\infty}(a_k) := \sup_{n \neq m} |a_n - a_m|\]
to be the \emph{diameter} of the sequence.
\end{definition}
These two statistics quantify convergence in that $\{ a_k \}$ converge iff
\[ N_\lambda(a_k) < \infty\]
for each $\lambda > 0$, and via the inequality
\[ \sup_{\lambda > 0} \, \lambda N_\lambda(a_k)^{1/r} \leq \mathcal{V}^r(a_k), \ 0 < r < \infty, \]
having finite $r$-variation implies that the sequence $\{ a_k\}$ converges.

For a sequence of functions, $\{ f_n(x) \} : X \to \mathbf{C}$, we define the functions
\[ N_\lambda(f_n)(x) := N_\lambda(f_n(x)), \; \; \; \mathcal{V}^r(f_n)(x) := \mathcal{V}^r(f_n(x)); \]
to see their utility in questions involving pointwise convergence, observe that a norm estimate of the form e.g.\
\[  \| \mathcal{V}^r(f_n) \|_{L^p(X,\mu)} \leq C\]
implies that $\mathcal{V}^r(f_n) <  \infty$ $\mu$-almost everywhere, and thus $\{ f_n\}$ converge almost everywhere as well. This approach was crucially used by Bourgain in \cite{B2}, which also allowed him to address the $L^{\infty}$-formulation of our main ergodic theoretic result, Theorem \ref{t:main0} below.

The relevant estimates for $\lambda N_{\lambda}^{1/2}, \mathcal{V}^r, \ r > 2$ derive, in many cases, from the following inequality, classically used as a convergence result in martingale theory \cite{LE}, with the restriction to $r > 2$ deriving from the law of the iterated logarithm; see \cite{JSW} for a discussion, and \cite{G+} or \cite{O+} for more exotic examples. Below, we let
\[ \mathbb{E}_k f := \sum_{|I| = 2^k \text{ dyadic}} \big( \frac{1}{|I|} \int_I f(t) \ dt \big) \mathbf{1}_I \]
denote the dyadic conditional expectation operators.

\begin{proposition} [L\'{e}pingle's Inequality, Special Case] \label{p-LEP}
There exists some absolute $0 < C < \infty$ so that the following estimate holds:
\[ \sup_{\lambda > 0} \| \lambda N_{\lambda}(\mathbb{E}_k f)^{1/2} \|_{L^2(\mathbb{R})} \leq C \|f\|_{L^2(\mathbb{R})}. \]
And, for each $r > 2$,
\[ \| \mathcal{V}^r(\mathbb{E}_k f) \|_{L^2(\mathbb{R})} \leq C  \frac{r}{r-2} \cdot \| f \|_{L^2(\mathbb{R})}. \]
\end{proposition}

The majority of the analytic part of our argument will concern a detailed analysis of the oscillatory nature of the $\{ \Xi_k \}$ at the $L^2(\mathbb{R})$-level, and moreover their smooth analogues,
\begin{align}\label{e:Phik}
    \widehat{\Phi_k f}(\beta) := \sum_{\theta \in \Theta} \varphi(2^k(\beta - \theta)) \hat{f}(\beta), \; \; \; \varphi \in \mathcal{C}_c^{\infty}(-1/2,1/2), \ \varphi(0) = 1
\end{align}
in $L^p(\mathbb{R})$. 
While the multi-frequency nature of the $\{ \Xi_k \}$ preclude a direct comparison with the martingale formulations, which gives a privileged role to the $0$-frequency (expectation), the projection structure of the $\{ \Xi_k \}$, namely
\[ \Xi_k \Xi_l \equiv \Xi_l, \; \; \; k \leq l\]
acts as a substitute; and similarly for the $\Phi_k$.
We summarize our main analytic results.

\begin{theorem}\label{t:L2}
There exists an absolute constant, $0 < C < \infty$, so that the following estimates holds:
for each $r > 2$,
\[ \| \mathcal{V}^r(\Xi_k f) \|_{L^2(\mathbb{R})} + \| \mathcal{V}^r(\Phi_k f : k \geq 1) \|_{L^2(\mathbb{R})} \leq C \log^2 N (\frac{r}{r-2})^2 \| f \|_{L^2(\mathbb{R})}. \]
And, for each $1 < p < \infty$ and each $\epsilon > 0$, there exists an absolute constant $0 < C_{\epsilon,p} < \infty$ so that
\[ \| \mathcal{V}^r(\Phi_k f : k \geq 1) \|_{L^p(\mathbb{R})} \leq C_{\epsilon,p} (\frac{r}{r-2})^2 N^{|1/p-1/2| + \epsilon} \| f \|_{L^p(\mathbb{R})}.\]
\end{theorem}

By applying the Multi-Frequency Calder\'{o}n-Zygmund Decomposition of Nazarov-Oberlin-Thiele \cite{NOT}, we are also able to address the endpoint.
\begin{theorem}\label{t:L1}
For each $\epsilon > 0$, there exists an absolute constant $0 < C_{\epsilon} < \infty$, an a single absolute constant, $0 < C < \infty$, so that
\[ \| \mathcal{V}^r(\Phi_k f : k \geq 1) \|_{L^{1,\infty}(\mathbb{R})} \leq \big( C (\frac{r}{r-2})^4 \log^4 N \cdot N^{1/2} + C_\epsilon N^{1/2+\epsilon} \big) \| f \|_{L^1(\mathbb{R})}.\]
\end{theorem}


Similar, if weaker, estimates were established in \cite{NOT,BKold} using interpolation, as was an analogue of Theorem \ref{t:L2} 
in \cite{DOP}; both estimates at the $L^2(\mathbb{R})$ level were on the order of $N^{c_r}$, and $N^{1/2 + c_r}$ at $L^{1,\infty}(\mathbb{R})$ level, with intermediate estimates established through interpolation. One feature of these arguments is that we establish Theorem \ref{t:L2} independently of our endpoint result.

With these analytic preliminaries in hand, the second half of our paper is devoted to Theorem \ref{t:main0}:

\begin{theorem}\label{t:main0}
For any $\sigma$-finite measure space equipped with a measure-preserving transformation, $T:X \to X$ and any $P \in \mathbb{R}[\cdot]$, the following ergodic averages converge almost everywhere whenever $f \in L^p(X), \ 1 < p \leq \infty$ if $P$ is linear, and $4/3 < p \leq \infty$ otherwise:
\[ \frac{1}{N} \sum_{n \leq N} T^{\lfloor P(n) \rfloor} f.\]
\end{theorem}

The main feature of our proof is that it is completely agnostic to any arithmetic information concerning $P$, save for the presence of at least one irrational coefficient; but this is no loss of generality, as the converse case was (first) addressed by Bourgain in \cite{B0}.

After completing this paper, we were informed by M. Wierdl that Theorem \ref{t:main0} can be shown to hold for all polynomials in the full expected range, $1 < p \leq \infty$ for Lebesgue spaces, i.e.\ complete probability spaces that are isomorphic to the ordinary Lebesgue measure space $([0,1], dx)$, by applying more involved, higher-dimensional analysis, namely the main result of \cite{MST}, see also \cite{MT}; we provide the argument in an appendix below, which was generously shared to us by him.

\subsection{Acknowledgements}
The author would like to thank Jan Fornal and Oleksiy Klurman for their help with Lemma \ref{l:gauss}, and M\'{a}t\'{e} Wierdl for making him aware of previous work on the matter.

\section{Preliminaries}\label{ss:not}
\subsection{Notation}
We let
\[ M_{\text{HL}}f(x) := \sup_{r > 0} \frac{1}{2r} \int_{|t| \leq r} |f(x-t)| \ dt  \]
denote the Hardy-Littlewood Maximal Function.

We introduce three measurements of oscillation for Hilbert-space valued Functions,
\[ \vec{f} := \{ f_\theta : \theta \in \Theta\}: \]
Regarding $\varphi \in \mathcal{C}_c^{\infty}(-1/2,1/2)$ as arbitrary but fixed, subject to the constraint that
\[ \sum_{0 \leq j \leq 10} \| \partial^j \varphi \|_{L^\infty(\mathbb{R})} \leq A_0,\]
we define the jump counting function, $\vec{N}_\lambda(\vec{f}) := \vec{N}_{\lambda,\varphi}(\vec{f})$ to be the supremum over all $M$ so that 
\begin{align}\label{e:vecJump}
    \text{there exists } k_0 < k_1 < \dots < k_M : \| \big((\varphi(2^{k_i} \cdot) - \varphi(2^{k_{i-1}} \cdot)) \widehat{f_\theta}\big)^{\vee} \|_{\ell^2(\Theta)} > \lambda;
\end{align}
we define
\begin{align}\label{e:vecVr}
    \mathcal{V}^r_{\Theta}(\vec{f}) := \mathcal{V}^r_{\Theta,\varphi}(\vec{f}) := 
    \sup \big( \sum_{i} \| \big((\varphi(2^{k_i} \cdot) - \varphi(2^{k_{i-1}} \cdot)) \widehat{f_\theta}\big)^{\vee} \|_{\ell^2(\Theta)}^r \big)^{1/r}
\end{align}
where the supremum runs over all finite increasing subsequences; 
and, we define
\begin{align}\label{e:vecMax}
    \mathcal{F}_{\Theta;K}(\vec{f})^2 := \mathcal{F}_{\Theta,\varphi;K}(\vec{f}) := \sum_{\theta \in \Theta} \sup_{k \geq K} |(\varphi(2^{k_i} \cdot) \widehat{f_\theta})^{\vee}|^2.
\end{align}
By \cite{KZK} and convexity, respectively the Fefferman-Stein inequalities, for each $1 < p < \infty$ there exists an absolute constant $0 < C_p < \infty$ so that we may estimate
\begin{align}
   \frac{r-2}{r} \| \mathcal{V}^r_{\Theta}(\vec{f}) \|_{L^p(\mathbb{R})} + \sup_K \| \mathcal{F}_{\Theta;K}(\vec{f}) \|_{L^p(\mathbb{R})}  \leq C_p A_0 \| \vec{f} \|_{L^p(\ell^2(\Theta))}, \; \; \; 1 < p < \infty.
\end{align}



\subsection{Asymptotic Notation}\label{sss:O}
We will make use of the modified Vinogradov notation. We use $X \lesssim Y$ or $Y \gtrsim X$ to denote
the estimate $X \leq CY$ for an absolute constant $C$ and $X, Y \geq 0.$  If we need $C$ to depend on a
parameter, we shall indicate this by subscripts, thus for instance $X \lesssim_p Y$ denotes the estimate $X \leq C_p Y$ for some $C_p$ depending on $p$. We use $X \approx Y$ as shorthand for $Y \lesssim X \lesssim Y$. We use the notation $X \ll Y$ or $Y \gg X$ to denote that the implicit constant in the $\lesssim$ notation is extremely large, and analogously $X \ll_p Y$ and $Y \gg_p X$.

We also make use of big-O notation and little-O: we let $O(Y)$  denote a quantity that is $\lesssim Y$ , and similarly
$O_p(Y )$ will denote a quantity that is $\lesssim_p Y$; we let $o_{t \to a}(Y)$
denote a quantity whose quotient with $Y$ tends to zero as $t \to a$ (possibly $\infty$), and
$o_{t \to a;p}(Y)$
denote a quantity whose quotient with $Y$ tends to zero as $t \to a$ at a rate depending on $p$.

\subsection{Elementary Oscillation Inequalities}
Given a sequence $\{ a_k \} \subset \mathbb{C}$, we begin by noting the trivial estimate, valid for each $r \geq 2$:
\[ \mathcal{V}^r(a_k) \leq  \mathcal{V}^2(a_k) \lesssim (\sum_k |a_k|^2)^{1/2}.\]

We next recall the following elementary inequality, first established in \cite{LL}, see also \cite{MT}, \cite[\S 4]{BOOK}
\begin{align}\label{e:V2} \mathcal{V}^2(a_k : k \leq 2^n) \leq 2 \sum_{m \leq n} \Big( \sum_{s < 2^{n-m}} |a_{s 2^m} + a_{(s+1)2^m}|^2 \Big)^{1/2}.
\end{align}

We will also make use of the following decomposition lemma from \cite{JSW}:
Given a disjoint partition, $[-M,M] = \bigcup_n I_n$, where $I_n = [a_n,b_n]$, we may bound
\begin{align}\label{e:split} \mathcal{V}^r(a_k) \lesssim \big( \sum_n \mathcal{V}^r(a_k : I_n)^r \big)^{1/r} + \mathcal{V}^r(b_n : n).
\end{align}

\section{The Proof of Theorem \ref{t:L2}}
The focus of this section is Theorem \ref{t:L2}, with the bulk of the work being devoted to the $L^2(\mathbb{R})$ theory. By monotone convergence, we can and will assume that all times are bounded by some absolute threshold,
\[ |k| \leq K_1,\]
provided our estimates are independent of $K_1$. Since
\begin{align}
    \sup_{k \geq 1} \| \Phi_k f \|_{L^p(\mathbb{R})} \lesssim_\epsilon N^{|1/p-1/2|+\epsilon} \| f \|_{L^p(\mathbb{R})}, \; \; \; 1 < p < \infty
\end{align}
by \cite{CRS}, and we of course have the trivial bound
\begin{align}
    \sup_{k \geq 1} \| \Phi_k f \|_{L^2(\mathbb{R})} + \sup_{k} \| \Xi_k f \|_{L^2(\mathbb{R})} \leq \| f \|_{L^2(\mathbb{R})},
\end{align}
in what follows we can and will assume that all scales $k$ considered satisfy the bounds
\[ |k| \geq 100 \log N;\]
so, below we will implicitly restrict to the regime of scales
\begin{align}
    K_0 := K_0(N) := 100 \log N \leq |k| \leq K_1.
\end{align}

Our first order of business is to reduce Theorem \ref{t:L2} to the smooth formulation. 

\begin{proof}[Proof of Theorem \ref{t:L2}, Reduction to Smooth Cut-Offs]
Assume the bound
\begin{align}
    \| \mathcal{V}^r(\Phi_k f : k \geq 1) \|_{L^2(\mathbb{R})} \lesssim (\frac{r}{r-2})^2 \log^2 N \| f \|_{L^2(\mathbb{R})}; 
\end{align}
by a square function argument, crucially using the bound
\[ \sup_\beta \sum_{k \geq 1} |\widehat{\Xi_k}(\beta) - \widehat{\Phi_k}(\beta)|^2 \lesssim 1,\]
after conflating the operators with their symbols, we deduce an analogous bound
\begin{align}\label{e:Xiscaled}
    \| \mathcal{V}^r(\Xi_k f : k \geq 1) \|_{L^2(\mathbb{R})} \lesssim (\frac{r}{r-2})^2 \log^2 N \| f \|_{L^2(\mathbb{R})}; 
\end{align}

It remains to address the bound
\begin{align}\label{e:Xiscaled}
    \| \mathcal{V}^r(\Xi_k f : k \leq 0) \|_{L^2(\mathbb{R})} \lesssim (\frac{r}{r-2})^2 \log^2 N \| f \|_{L^2(\mathbb{R})}; 
\end{align}

By applying \eqref{e:V2} pointwise to the sequence $\{ \Xi_kf(x) \}$ and using orthogonality in $L^2(\mathbb{R})$, we may also efficiently address the $r$-variation of $\{ \Xi_k f\}$ restricted to finitely many scales:
\begin{align}\label{e:V2} 
\| \mathcal{V}^2(\Xi_k f : k \in E)\|_{L^2(\mathbb{R})} \lesssim \log |E| \cdot \|  f \|_{L^2(\mathbb{R})}.
\end{align}

So, let 
\begin{align}\label{e:E} E := \{ |k| \leq M : \text{ there exists } \theta \neq \theta' \in \Lambda : 2^{-k - N^{10}} \leq |\theta - \theta'| \leq 2^{-k + N^{10}} \} \end{align}
which has size $|E| \lesssim N^{15}$; we estimate the contribution to the jump counting function coming from times in $E$ using \eqref{e:V2}.

Now, let $I_n := [a_n,b_n]$ be such that for each $k \in I_n$, $R_k$ has $n$ connected components, and $[a_n,b_n] \cap E = \emptyset$. 

Note that since $I_n \cap E = \emptyset$, if for $k \in I_n$ we express $R_k$ as a disjoint union of intervals,
\[ R_k = \bigcup_{l \leq n} (c_l-2^{-k},d_l+2^{-k}) = \bigcup_{l \leq n} \big( \frac{c_l+d_l}{2} + (-2^{-k} - \frac{d_l - c_l}{2}, 2^{-k} + \frac{d_l-c_l}{2}) \big);\]
while the radii of our intervals are not quite dyadic, this introduces only minor notational changes by construction of $E$.

We apply \eqref{e:split} to majorize
\begin{align}
    \mathcal{V}^r(  \Xi_k f :k) \lesssim \big( \sum_{n \leq N} \mathcal{V}^r( \Xi_k f : k \in I_n)^2 \big)^{1/2} + \mathcal{V}^2(\Xi_{b_n} f : n \leq N),
\end{align}
and bound
\[ \| \mathcal{V}^2(\Xi_{b_n}f : n \leq N) \|_{L^2(\mathbb{R})} \lesssim \log N \|f \|_{L^2(\mathbb{R})}\]
by \eqref{e:V2}, so we focus on the first term. But, we have the equality
\begin{align}
\mathcal{V}^r( \Xi_k f : k \in I_n) \equiv \mathcal{V}^r( \Xi_k ( \Xi_{a_n} f - \Xi_{b_n} f ) : k \in I_n) 
\end{align}
so integrating yields
\begin{align}
    \| \big( \sum_{n \leq N} \mathcal{V}^r( \Xi_k f : k \in I_n)^2 \big)^{1/2} \|_{L^2(\mathbb{R})}^2 &\lesssim (\frac{r}{r-2})^4 \log^4 N \cdot \sum_n \| \Xi_{a_n} f - \Xi_{b_n} f \|_{L^2(\mathbb{R})}^2 \\
    & \qquad \lesssim (\frac{r}{r-2})^4 \log^4 N \| f \|_{L^2(\mathbb{R})}^2.
\end{align}
Above, we used a re-scaling of the estimate
\[ \| \mathcal{V}^r(\Xi_k g : k \geq 1 ) \|_{L^2(\mathbb{R})} \lesssim (\frac{r}{r-2})^2 \log^2 N \| g \|_{L^2(\mathbb{R})}.\]
\end{proof}

So, in what follows, our work on Theorem \ref{t:L2} will concern only the smooth formulation, namely
\[ \mathcal{V}^r(\Phi_k f).\]
We begin by adapting a metric chaining argument of \cite{Bourbaki}, which we now describe; below, we set $ 
    \phi_k := \varphi(2^k \cdot)^{\vee}.$

\medskip

For each $x \in \mathbb{R}$, we set
\begin{align}\label{e-ftheta} X(x) := \{ \phi_k*\vec{f_\Theta}(x) : k \} := \{ \big( \phi_k*f_{\theta_1}(x),\dots,\phi_k*f_{\theta_N}(x) \big) : k \} \end{align}
and set
\begin{align*}
\vec{N}_\lambda(x) := \vec{N}_{\lambda}(\vec{f}_{\Theta})(x),
\end{align*}
see \eqref{e:vecJump}. For each $v$ so that 
\[ 2^{-v} \leq   \text{diam}(X(x)) \leq 2 \mathcal{F}_{\Theta}(x) := 2\mathcal{F}_{\Theta;K_0}(\vec{f}_{\Theta})(x),
\]
see \eqref{e:vecMax}, define $\Lambda_v(x) \subset [K_0,K_1]$ to be a collection of times $t$ so that
\begin{align}
X(x) \subset \bigcup_{t \in \Lambda_v(x)} B\big( \phi_t*\vec{f_\Theta}(x), 2^{-v} \big),
\end{align}
subject to the constraint that $|\Lambda_v(x)|$ is minimal; the cardinality is essentially the $2^{-v}$-\emph{entropy} of the set, and note that for $v$ in the proscribed range, we may bound
\[ |\Lambda_v(x)| \leq 2 \vec{N}_{2^{-v}}(x).\]
Above,
\[ B(\phi_k*\vec{f_\Theta}(x),2^{-v}) := \{ (b_{\theta_1},\dots,b_{\theta_N}) : \| b_{\theta_n} - \phi_k*f_{\theta_n}(x) \|_{\ell^2(\Theta)} \leq 2^{-v} \} \]
are balls with respect to the $\ell^2(\Theta)$-norm.

For each $t \in \Lambda_v(x)$, define the \emph{parent} of $t$, $\varrho(t)\in \Lambda_{v-1}(x)$ to be the minimal element so that 
\begin{align}\label{e-int}
B\big( \phi_t*\vec{f_\Theta}(x), 2^{-v} \big) \cap B \big( \phi_{\varrho(t)}*\vec{f_\Theta}(x),2^{1-v} \big) \neq \emptyset.
\end{align}

With this construction in hand, we begin with localized $L^2(\mathbb{R})$-estimates, namely the following lemma. For the remainder of this section, we set
\begin{align}
\widehat{f_\theta}(\beta) := \chi(\beta) \hat{f}(\beta + \theta)
\end{align}
where $\mathbf{1}_{[-1/4,1/4]} \leq \chi \leq \mathbf{1}_{[-1/2,1/2]}$ is smooth.

\begin{lemma}\label{l:localized}
    Whenever $|I| = 10$, for any $\frac{2}{r} < u < 1$, we may bound 
    \begin{align}\label{e:localized}
    \| \mathcal{V}^r(\Phi_k f) \|_{L^2(I)} &\lesssim \frac{1}{1-u} \log N \min_{x_I \in I} \, \mathcal{V}_{\Theta}^s(x_I) + N^{-10} \min_{x_I \in I} \, \mathcal{F}_\Theta(x_I)^{1-u} \cdot \mathcal{V}_{\Theta}^{ru}(x_I)^u \\
    & \qquad \qquad \qquad + N^{-10} \min_{x_I \in I} \, (\sum_\theta M_{\text{HL}} f_\theta^2)^{1/2}(x_I).
\end{align}
where $\mathcal{V}_\Theta^r := \mathcal{V}^r_{\Theta}(\vec{f_\Theta})$, see \eqref{e:vecVr}, and $2 < s < r$ can be chosen to satisfy $\frac{1}{s-2} \lesssim \frac{\log N}{r-2}.$
\end{lemma}
Aside from its apparent use for developing our $L^2(\mathbb{R})$ theory, this Lemma \ref{l:localized} is useful for proving low-$L^p(\mathbb{R})$ estimates, as for any $1 < p \leq 2$ we may bound
\begin{align}
\| \mathcal{V}^r(\Phi_k f)\|_{L^p(\mathbb{R})}^p &= \sum_{|I| = 10} \| \mathcal{V}^r(\Phi_k f)\|_{L^p(I)}^p \\
&  \lesssim (\frac{1}{1-u})^p \log^p N \sum_{|I| = 10} \min_{x_I \in I} \, \mathcal{V}_{\Theta}^s(x_I)^p + N^{-10} \sum_{|I| = 10} \min_{x_I \in I}  \, \mathcal{F}_\Theta(x_I)^{p(1-u)} \cdot \mathcal{V}_{\Theta}^{ru}(x_I)^{pu} \\
&  \qquad \qquad + N^{-10} \sum_{|I| = 10} \min_{x_I \in I}  \,  (\sum_\theta M_{\text{HL}} f_\theta^2)^{1/2}(x_I)^p \\
&  \qquad \lesssim (\frac{1}{1-u})^p \log^p N \| \mathcal{V}_{\Theta}^s \|_{L^p(\mathbb{R})}^p + N^{-10} \| \mathcal{F}_{\Theta}^{(1-u)} (\mathcal{V}_{\Theta}^{ru})^u \|_{L^p(\mathbb{R})}^p \\
& \qquad \qquad \qquad + N^{-10} \| (\sum_\theta M_{\text{HL}} f_\theta^2) \|_{L^p(\mathbb{R})}^p;
\end{align}
taking $p$th roots, we bound
\begin{align}\label{e:usethis}
\| \mathcal{V}^r(\Phi_k f)\|_{L^p(\mathbb{R})} &\lesssim (1-u)^{-1} \frac{r}{r-2} \log^2 N \| (\sum_\theta |f_\theta|^2)^{1/2} \|_{L^p(\mathbb{R})} + N^{-10} \| \mathcal{F}_{\Theta}^{1-u} \|_{p/1-u} \| (\mathcal{V}_{\Theta}^{ru})^u \|_{p/u} \\
& \qquad \lesssim (1-u)^{-1} \frac{r}{r-2} \log^2 N \| (\sum_\theta |f_\theta|^2)^{1/2} \|_{L^p(\mathbb{R})} + N^{-10} \| \mathcal{F}_{\Theta} \|_p^{1-u} \| \mathcal{V}_{\Theta}^{ru} \|_p^u \\
& \qquad \qquad \lesssim \big( (1 - u)^{-1} \frac{r}{r-2} \log^2 N + N^{-10} \frac{r}{ru-2} \big) \| (\sum_\theta |f_\theta|^2)^{1/2} \|_{L^p(\mathbb{R})} \\
& \qquad \qquad \qquad \lesssim (\frac{r}{r-2})^2 \log^2 N \| (\sum_\theta |f_\theta|^2)^{1/2} \|_{L^p(\mathbb{R})},
\end{align}
after optimizing in $\frac{2}{r} < u < 1$.

This immediately implies Theorem \ref{t:L2} at the $L^2(\mathbb{R})$ level, and for $1 < p < 2$, the result follows since we may bound
\begin{align}
    \| (\sum_\theta |f_\theta|^2)^{1/2} \|_{L^p(\mathbb{R})}  \lesssim_p N^{1/p-1/2 + \epsilon} \|f \|_{L^p(\mathbb{R})}
\end{align}
by randomization and \cite{CRS}.

The high $L^p(\mathbb{R})$ theory, $2 < p < \infty$, is easier, as we may bound
\begin{align}
    \mathcal{V}^r(\Phi_k f) \leq N^{1/2} (\sum_\theta \mathcal{V}^r(f_\theta)^2)^{1/2}
\end{align}
where 
\[ \mathcal{V}^r(g) := \sup \big( \sum_i |\phi_{k_i}*g - \phi_{k_{i-1}}*g|^r\big)^{1/r} \]
with the supremum running over all finite increasing subsequences. By duality, we estimate
\begin{align}
    \| \mathcal{V}^r(\Phi_k f) \|_{L^p(\mathbb{R})}^2 \leq N \| \sum_{\theta} \mathcal{V}^r(f_\theta)^2 \|_{L^{p/2}(\mathbb{R})} = N \sum_\theta \int \mathcal{V}^r(f_\theta)^2 \cdot w  
\end{align}
for some non-negative $w \geq 0$ with $\| w \|_{L^{(p/2)'}(\mathbb{R})} = 1$. If we choose $1 < t < (p/2)'$ sufficiently close to $1$, then 
\[ \| w_t \|_{L^{(p/2)'}(\mathbb{R})} := \| (M_{\text{HL}} w^t)^{1/t} \|_{L^{(p/2)'}(\mathbb{R})} \lesssim_{t,p} 1, \]
and $w_t \in A_1(\mathbb{R})$ is an $A_1$ weight, see \cite{DU}. By \cite{KZK}, we may therefore bound
\[ \int \mathcal{V}^r(f_\theta)^2 \cdot w \leq \int \mathcal{V}^r(f_\theta)^2 \cdot w_t \lesssim (\frac{r}{r-2})^2 \int |f_\theta|^2 \cdot w_t;\]
putting everything together, we bound
\begin{align}
\| \mathcal{V}^r(\Phi_k f) \|_{L^p(\mathbb{R})}^2 &\lesssim (\frac{r}{r-2})^2 N \int \sum_\theta |f_\theta|^2 \cdot w_t \lesssim (\frac{r}{r-2})^2 N \| \sum_\theta |f_\theta|^2 \|_{L^{p/2}(\mathbb{R})} \| w_t \|_{L^{(p/2)'}(\mathbb{R})} \\
& \qquad \lesssim (\frac{r}{r-2})^2 N \| (\sum_\theta |f_\theta|^2 )^{1/2} \|_{L^p(\mathbb{R})}^2 \lesssim (\frac{r}{r-2})^2 N \| f \|_{L^p(\mathbb{R})}^2,
\end{align}
with the final inequality being a consequence of Rubio de Francia's square function estimates \cite{RdF}; the result is concluded by interpolation.

Therefore, Theorem \ref{t:L2} will be established once we have proven Lemma \ref{l:localized}.
\begin{proof}[Proof of Lemma \ref{l:localized}]
Let $I$ be an arbitrary interval of length $10$; by translation invariance we may assume that $I = [0,10]$.
Since we have restricted to scales $k \geq K_0$, whenever $0 \leq x,y,z \leq  10$, we may bound
\begin{align}
    \sum_{\theta} e(\theta x) \phi_k*f_\theta(x) =     \sum_{\theta} e(\theta x) \phi_k*f_\theta(y) + N^{-10} 2^{-k/10} \cdot \big( \sum_{\theta} M_{\text{HL}} f_\theta(z)^2 \big)^{1/2}
\end{align}
using the smoothness of $\phi$. So, for any $0 \leq y,z \leq 10 $
\begin{align}
\mathcal{V}^r( \Phi_k f)(x) \leq \mathcal{V}^r\big( \sum_\theta e(\theta x) \phi_k*f_\theta(y) : k ) + N^{-10} \cdot \big( \sum_{\theta} M_{\text{HL}} f_\theta(z)^2 \big)^{1/2}.   
\end{align}
With $y$ to be determined below, we apply the metric chaining mechanism to the set $X(y)$, setting
\[ \psi_t := \phi_t - \phi_{\varrho(t)},\]
to bound
\begin{align}
    &\mathcal{V}^r\big( \sum_\theta e(\theta x) \phi_k*f_\theta(y) : k ) \leq \sum_{2^{-v} \leq 2 \mathcal{F}_{\Theta}(y)} \mathcal{V}^r\big( \sum_{\theta} e(\theta x) \psi_t*f_\theta(y) : t \in \Lambda_v(y) \big) \\
    & \qquad \qquad + N^{-10} \cdot \big( \sum_{\theta} M_{\text{HL}} f_\theta(z)^2 \big)^{1/2} \\
    & \qquad \lesssim \sum_{2^{-v} \leq 4 \mathcal{F}_{\Theta}(y)} \big( \sum_{t \in \Lambda_v(y)} \big| \sum_{\theta} e(\theta x) \psi_t*f_\theta(y)|^r \big)^{1/r}   +N^{-10} \cdot \big( \sum_{\theta} M_{\text{HL}} f_\theta(z)^2 \big)^{1/2}.  
\end{align}
For each $v$, we bound
\begin{align}
    \| (\sum_{t \in \Lambda_v(y)} \big| \sum_{\theta} e(\theta x) \psi_t*f_\theta(y) \big|^r)^{1/r} \|_{L^2_x([0,10])} &\lesssim 2^{-v} \min\{ N^{1/2} |\Lambda_v(y)|^{1/r}, |\Lambda_v(y)|^{1/2} \} \\
    & \qquad \lesssim 2^{-v} \min\{ N^{1/2}\vec{N}_{2^{-v}}^{1/r}(y), \vec{N}_{2^{-v}}^{1/2}(y)\}:
\end{align} 
the first inequality is just a pointwise estimate, which follows from applying Cauchy Schwartz in the inner sum in $\theta$; for the second, we bound
\begin{align}
    \| (\sum_{t \in \Lambda_v(y)} \big| \sum_{\theta} e(\theta x) \psi_t*f_\theta(y) \big|^r)^{1/r} \|_{L^2([0,10])} &\leq \| (\sum_{t \in \Lambda_v(y)} \big| \sum_{\theta} e(\theta x) \psi_t*f_\theta(y) \big|^2)^{1/2} \|_{L^2([0,10])} \\
    & \qquad \lesssim 2^{-v} |\Lambda_v(y)|^{1/2} \lesssim 2^{-v} \vec{N}_{2^{-v}}^{1/2}(y),
\end{align}
using the fact that $\{ \theta \}$ are $1$ separated, and the elementary inequality
\begin{align}
    \| \sum_\theta e(\theta x) a_\theta \|_{L^2([0,10])} \lesssim \| \sum_\theta e(\theta x) a_\theta w(x) \|_{L^2(\mathbb{R})} \lesssim (\sum_\theta |a_\theta|^2)^{1/2}
\end{align}
where $\mathbf{1}_{[0,10]} \leq w \leq \mathbf{1}_{[-10,20]}$ has a Fourier transform supported inside $(-1/2,1/2)$.

In total, for any $0 \leq y \leq 10$, with $s$ as above and $A = \frac{100}{1-u}$, we have bounded
\begin{align}
    &\sum_{2^{-v} \leq 2 \mathcal{F}_{\Theta}(y)} \| \mathcal{V}^r\big( \sum_{\theta} e(\theta x) \psi_t*f_\theta(y) : t \in \Lambda_v(y) \big) \|_{L^2_x([0,10])} \leq \sum_{2^{-v} \leq 2 \mathcal{F}_{\Theta}(y)} 2^{-v} \min\{ N^{1/2} \vec{N}_{2^{-v}}^{1/r}(y),\vec{N}_{2^{-v}}^{1/2}(y) \} \\
    & \qquad \leq \sum_{\mathcal{F}_{\Theta}(y)/N^A \leq 2^{-v} \leq 2\mathcal{F}_{\Theta}(y)} 2^{-v} \vec{N}_{2^{-v}}^{1/s}(y) + N^{1/2} \sum_{2^{-v} \leq \mathcal{F}_{\Theta}(y)/N^A} 2^{-v(1-u)} \big(2^{-v}  \vec{N}_{2^{-v}}^{1/ru}(y) \big)^{u} \\
    & \qquad \qquad \lesssim \frac{1}{1-u} \cdot \log N \cdot \mathcal{V}_{\Theta}^s(y) + N^{-10} \cdot \mathcal{F}_{\Theta}(y)^{(1-u)} \cdot \mathcal{V}_{\Theta}^{ru}(y)^u, 
\end{align}
completing the proof.
\end{proof}

We now address endpoint estimates.

\section{The Proof of Theorem \ref{t:L1}}
In this section, we apply Nazarov-Oberlin-Thiele's Multi-frequency Calder\'{o}n-Zygmund decomposition \cite{NOT}, to establish Theorem \ref{t:L1}, restated for convenience below. 
\begin{proposition}\label{p:1scale}
    For each $r > 2, \ \epsilon >0$,
    \[ \| \mathcal{V}^r(\Phi_k f) \|_{L^{1,\infty}(\mathbb{R})} \lesssim_\epsilon \big( (\frac{r}{r-2})^4 \log^4 N \cdot N^{1/2} + N^{1/2+\epsilon} \big) \cdot \| f \|_{L^1(\mathbb{R})}. \]
\end{proposition}

The proof goes by way of a clever decomposition lemma of \cite{NOT}.
\begin{lemma}[Multi-Frequency Calder\'{o}n-Zygmund Decomposition]\label{l:mfcz}
Let $\Theta := \{ \xi_n : n \leq N \} \subset \mathbb{R}$ be an arbitrary but fixed collection of frequencies. Then, for any $f \in L^1$, there exists a decomposition
\[ f = g + \sum_I b_I\]
   where $\{ I \}$ are a disjoint collection of intervals with
   \[ \sum_I |I| \lesssim N^{1/2} \cdot \| f \|_{L^1(\mathbb{R})},\] so that the decomposition satisfies the following properties:
   \begin{itemize}
       \item $\|f \cdot \mathbf{1}_I \|_{L^1(\mathbb{R})} =: \| f_I \|_{L^1(\mathbb{R})}  \lesssim |I|/N^{1/2} $;
       \item $\| g_I \|_{L^2(\mathbb{R})} =: \| f_I - b_I \|_{L^2(\mathbb{R})} \lesssim |I|^{1/2}$;
       \item $\| g \|_{L^2(\mathbb{R})}^2 \lesssim N^{1/2} \cdot \| f \|_{L^1(\mathbb{R})}$;
       \item $b_I$ is supported in $3I$ and $\| b_I \|_{L^1(\mathbb{R})} \lesssim |I|$; and
       \item For each $I$, $\int b_I(x) e(-\xi x) \ dx = 0$ for all $\xi \in \Theta$.
   \end{itemize}
\end{lemma}

We introduce two smooth functions, $\rho$, which satisfies $\mathbf{1}_{[-5,5]} \leq \rho \leq \mathbf{1}_{[-10,10]}$, and $\eta$, with integral one and compact support inside $(-1/10,1/10)$. The key feature we will use is that
\[ \rho_0 \cdot \varphi := (\eta* \rho) \cdot \varphi \equiv \varphi.\]

We sparsify our collection of frequencies, $\Theta$, into (say) $100$ subcollections, so that
\[ \{ \rho_0(2^j(\beta - \theta)) : \theta \}  \]
are disjointly supported, for each $j \geq 1$.
We use windowed Fourier series to arrive at a more convenient representation
\[ \Phi_j f = \sum_{l \in \mathbb{Z}} \Pi_{j,l} f \]
where
\[ \Pi_{j,l} f(x) := \sum_{\theta}
e(\theta x) \rho_0^{\vee}(2^{-j} x -l) \Big( \int f(s) e(-\theta s) 2^{-j} \varphi^{\vee}(2^{-j} s-l) \ ds \Big).\]

\begin{proof}[Proof of Proposition \ref{p:1scale}]
We apply the multi-frequency Calder\'{o}n-Zygmund decomposition with respect to the frequencies $\Theta,$
and set 
\begin{align}\label{e:set} E := \bigcup_I N^\epsilon I. \end{align}
By homogeneity, it suffices to estimate
\[ |\{ \mathcal{V}^r(\Phi_k f)| \gg 1\}| \lesssim_\epsilon \big( (\frac{r}{r-2})^4 \log^4 N \cdot N^{1/2} + N^{1/2+\epsilon} \big) \cdot \| f \|_{L^1(\mathbb{R})}. \]
We estimate
\begin{align}
    |\{ \mathcal{V}^r(\Phi_k f) \gg 1\}| &\leq |\{ \mathcal{V}^r(\Phi_k g) \gg 1\}| + |E| + \sum_I  \sum_j \| \Phi_j b_I \|_{L^1(I^*)} \\
    & \qquad \lesssim (\frac{r}{r-2})^4 \log^4 N \cdot N^{1/2} \|f \|_{L^1(\mathbb{R})} + N^{1/2+\epsilon} \|f \|_{L^1(\mathbb{R})} + \sum_I \sum_j \| \Phi_j b_I \|_{L^1(I^*)},
\end{align}
where $I^* := (N^\epsilon I)^c$. We will show that 
    \begin{align}\label{e:triangle}\sum_j \| \Pi_j b_I \|_{L^1(I^*)} \lesssim |I|.\end{align}
When $2^j \leq |I|$, we use the Schwartz decay of our bump functions to estimate the small-scale contribution
\begin{align}
    \sum_{2^j \leq |I|} \| \Pi_j b_I \|_{L^1(I^*)} &\lesssim_M \sum_{2^j \leq |I| } N \| b_I \|_{L^1(\mathbb{R})} \sup_{y \in 3I}\| 2^{-j} (1 + 2^{-j} |x-y|)^{-M} \|_{L^1(I^*)} \\
    & \qquad \lesssim |I|.
\end{align}

For the large scales, we pull out the sum in $l$, and estimate
\begin{align}
\| \Pi_{j,l} b_I \|_{L^1(\mathbb{R})} \lesssim (1  + |l|)^{-2} 2^{-j} |I|.
\end{align}

We bound
\begin{align}
    \| \Pi_{j,l} b_I \|_1 &\lesssim 2^{j/2} \| \sum_{\theta} e(\theta x) \rho^{\vee}(2^{-j} x -l) \big( \int b_I(s) e(-\theta s) 2^{-j} \varphi^{\vee}(2^{-j} s-l) \ ds \big) \|_{L^2(\mathbb{R})} \\
    & \qquad \lesssim 2^{j/2} \| \big( \sum_{\theta} \big| \int b_I(s) e(-\theta s) 2^{-j} \varphi^{\vee}(2^{-j} s-l) \ ds \big|^2 |\rho^{\vee}(2^{-j} x -l)|^2 \big)^{1/2} \|_{L^2(\mathbb{R})} \\
    & \qquad \qquad \lesssim 2^j \big(\sum_{\theta} \big| \int b_I(s) e(-\theta s) 2^{-j} \varphi^{\vee}(2^{-j} s-l) \ ds \big|^2 \big)^{1/2},
\end{align}
where we used orthogonality in passing to the square function formulation in the second line.

Thus, it suffices to prove that 
\begin{align}
 \sum_{\theta} \big| \int b_I(s) e(-\theta s) 2^{-j} \varphi^{\vee}(2^{-j} s-l) \ ds \big|^2 \lesssim (1 + |l|)^{-10} \cdot 2^{-4j} |I|^2.
\end{align}

Since each $b_I$ is orthogonal to the frequencies in $\Theta$, we can express
\begin{align}
     &\int b_I(s) e(-\theta s) 2^{-j} \varphi^{\vee}(2^{-j} s-l) \ ds \\
     & \qquad \equiv 
\int b_I(s) e(-\theta s) 2^{-j} \big( \varphi^{\vee}(2^{-j} s-l) - \varphi^{\vee}(-l) \big) \ ds  =: \int b_I(s) e(-\theta s) \psi_{j,l}(s) \ ds
\end{align}
    where $\psi_{j,l}$ satisfies the differential inequalities
    \begin{align}
2^{2j} \big( \|\psi_{j,l}\|_{L^\infty(\mathbb{R})} + \|\partial \psi_{j,l}\|_{L^\infty(\mathbb{R})} \big) + 2^{3j} \| \partial^2 \psi_{j,l} \|_{L^\infty(\mathbb{R})} \lesssim (1 + |l|)^{-10}
    \end{align}
using the Mean-Value Theorem and the compact support of $\varphi$.

So, our task is to prove
\begin{align}
    \sum_{\theta} \big| \int b_I(s) \psi_{j,l}(s) e(-\theta s) \ ds \big|^2 \lesssim (1 + |l|)^{-10} \cdot 2^{-4j} |I|^2.
\end{align}
If we replace $b_I$ with $f_I$ in the above, then we may use fairly crude estimates: 
\begin{align}
\Big| \int f_I(s) \psi_{j,l}(s) e(-\theta s) \ ds \Big| \leq (1 + |l|)^{-10} \cdot 2^{-2j} \cdot \| f_I \|_{L^1(\mathbb{R})} \lesssim (1 + |l|)^{-10} \cdot 2^{-2j} N^{-1/2} |I|.
\end{align}

It remains to address
\[ \sum_{\theta} \big| \int g_I(s) \psi_{j,l}(s) e(-\theta s) \ ds \big|^2; \]
    we use duality. Specifically, for an appropriate sequence $\{ a_{\theta} \}$ with $\sum_{\theta} |a_{\theta}|^2 \leq 1,$
we bound
\begin{align}
    \sum_{\theta} \Big| \int g_I(s) \psi_{j,l}(s) e(-\theta s) \ ds \Big|^2 &\leq \Big| \int g_I(s) \big( \sum_{\theta} a_{\theta} \psi_{j,l}(s) e(-\theta s) \big) \ ds \Big|^2 \\
    & \qquad \lesssim \| g_I \|_{L^2(\mathbb{R})}^2 \| \sum_{\theta} a_{\theta} \psi_{j,l}(s) e(-\theta s) \chi_I(s) \|_{L^2(\mathbb{R})}^2 \\
    & \qquad \qquad \lesssim |I| \| \sum_{\theta} a_{J(\omega)} \psi_{j,l}(s) e(-\theta s) \chi_I(s) \|_{L^2(\mathbb{R})}^2
\end{align}
where $\mathbf{1}_{10I } \leq \chi_I \leq \mathbf{1}_{20 I}$
is smooth, and in particular satisfies
\[ g_I \chi_I \equiv g_I.\]
    By integration by parts twice, we may bound
\begin{align}
    \int \psi_{j,l}(s) \overline{\psi_{j,l}}(s) e(-(\theta - \theta' )s) |\chi_I(s)|^2 \ ds \lesssim (1 + |l|)^{-10} \cdot 2^{-4j}  (1 + |\theta - \theta'|)^{-2} |I|,
\end{align}
and since $|\theta - \theta'| \geq 100$ (by our initial sparsification), the result follows from Schur's test:

\begin{align}
&\| \sum_{\theta} a_{\theta} \psi_{j,l}(s) e(-\theta s) \chi_I(s) \|_{L^2(\mathbb{R})}^2 \lesssim (1 + |l|)^{-10} \cdot 2^{-4j} |I| \sum_{\theta,  \theta'} |a_{\theta}| |a_{\theta'}| (1 + |\theta - \theta'|)^{-2} \\
& \qquad \lesssim (1 + |l|)^{-10} \cdot 2^{-4j} |I| .
    \end{align}
\end{proof}

With our analytic rsults in hand, we now our attention to Theorem \ref{t:main0}, namely the issue of pointwise convergence of ergodic averages.

\section{Pointwise Convergence of Fractional Parts of Real-Variable Polynomials}

In this section we prove our main result in pointwise ergodic theory, with $L^{\infty}(X)$-result established by Bourgain in \cite[\S 8]{B2}. Below, we free the parameter $N$ to denote a time, $N \in \mathbb{N}$; typically we will restrict to a member of a lacunary subsequence. We set
\begin{align}\label{e:pd}
p_0 := \begin{cases} 1 & \text{ if $P$ is linear } \\
\frac{4}{3} & \text{ otherwise}.\end{cases}
\end{align}
We recall our result:
\begin{theorem}\label{t:ptwise0}
    Let $P(t) \in \mathbb{R}[\cdot]$ be a real-variable polynomial, and let $(X,\mu,T)$ be a $\sigma$-finite measure-preserving system. Then
    \[ \frac{1}{N} \sum_{n \leq N} T^{\lfloor P(n) \rfloor} f\]
    converges $\mu$-a.e.\ for all $f \in L^p(X), \ p_0 < p \leq \infty$. 
\end{theorem}
By Bourgain's polynomial ergodic theorem \cite{B2}, it suffices to consider the case where
\[ P(t) = \sum_{0 \leq j \leq d} b_j t^j =: b_0 + Q(t) \]
contains at least one irrational non-constant coefficient; henceforth, we will fix such a $P$, and allow all estimates to depend on $\{ b_j \}$, so in particular for the remainder of the paper we reserve $d := \text{deg}(P)$; and, we will \emph{not} assume that $b_0 = 0$.

The proof of Theorem \ref{t:ptwise0} follows a similar stategy to Bourgain's work on polynomial ergodic averages \cite{B2}, with key differences arising at the level of structure of the major arcs: because of the irrationality of the coefficients of $P$, the arithmetic techniques pioneered by Bourgain \cite{B2} and later Ionescu-Wainger \cite{IW} used to address polynomial ergodic averages do not apply. We instead rely on coarser analytic statistics of complete exponential sums, namely Lemma \ref{l:gauss} below.

Our proof goes by way of several reductions. We begin by recalling a standard Whitney decomposition to replace the rough cut-off, $\mathbf{1}_{[0,1)}$ below:
\begin{align} \frac{1}{N} \sum_{n \leq N} T^{\lfloor P(n) \rfloor} f \equiv \sum_k \frac{1}{N} \sum_{n \leq N} \mathbf{1}_{[0,1)}(P(n) - k) T^k f. \end{align}

We summarize the construction:
\begin{lemma}[Whitney Decomposition]
    There exist a collection of dyadic intervals $\{ J \}$ so that
    \begin{itemize}
        \item $100 J \subset [0,1)$;
        \item $\bigcup J = [0,1)$ is a disjoint union;
        \item The $20$-fold dilates of $\{ J \}$ have bounded overlap:
        \[ \sum \mathbf{1}_{20 J} \lesssim 1;\]
        \item There exists an absolute $C = O(1)$ so that 
        \[ |\{ J : |J| = 2^{-n} \}| \leq C.\]
    \end{itemize}
\end{lemma}    
Associated to our dyadic Whitey decomposition, we define a family of smooth multipliers,
   \[ \{ \phi_{J} \} \]
   satisfying natural derivative conditions
   \begin{align}\label{e:deriv} \sup_{J, \ \alpha \leq 10} |J|^{\alpha} \cdot |\partial^\alpha  \phi_J| \lesssim 1,\end{align}
so that
\[ \mathbf{1}_{[0,1)} = \sum \phi_{J}\]
and \[\text{supp } \phi_{J} \subset J' \]
is an interval that's concentric with $J$ with dyadic side length, satisfying $|J'| = 4|J|$.

We decompose
\begin{align}
    &\sum_k \frac{1}{N} \sum_{n \leq N} \mathbf{1}_{[0,1)}(P(n) - k) T^k f \\
    & \qquad = \sum_J \big( \sum_k \frac{1}{N} \sum_{n \leq N} \phi_J(P(n) - k) T^k f \big).
\end{align}

At the $L^{\infty}(X)$ endpoint, we have a cheap estimate, which follows from Weyl's equidistribution criterion and the continuity of each $\phi_J$, using crucially the fact that $Q$ has at least one irrational coefficient:
\begin{align}
    \limsup |\sum_k \frac{1}{N} \sum_{n \leq N} \phi_J(P(n) - k) T^k f| \leq \| \phi_J \|_{L^1(\mathbb{T})} \| f \|_{L^{\infty}(X)} \lesssim |J| \| f \|_{L^{\infty}(X)} 
\end{align}
In particular
\begin{align}\label{e:Linf}
    \| \mathcal{V}^{\infty}\big( \sum_k \frac{1}{N} \sum_{n \leq N} \phi_J(P(n) - k) T^k f \big) \|_{L^{\infty}(X)} \lesssim |J| \| f \|_{L^{\infty}(X)} .
\end{align}

We will complement this by proving that for each $R \in \mathbb{N}$ and each $r > 2, \ p_0 < p < \infty$, 
\begin{align}
     \| \mathcal{V}^r\big( \sum_k \frac{1}{N} \sum_{n \leq N} \phi_J(P(n) - k) T^k f \big) \|_{L^p(X)} \lesssim_{p,r,R}  \| f \|_{L^p(X)},
\end{align}
where $I_R := \{ \lfloor 2^{k/R} \rfloor : k \geq 1 \}.$ By mixed norm interpolation, this will imply that for each $p_0 < p < \infty$, there exists $2 < r = r(p) < \infty$, \ $ \epsilon = \epsilon(p) > 0$ so that 
\begin{align}
    \| \mathcal{V}^r\big( \sum_k \frac{1}{N} \sum_{n \leq N} \phi_J(P(n) - k) T^k f : N \in I_R \big) \|_{L^p(X)} \lesssim_{p,R} |J|^{\epsilon} \| f \|_{L^p(X)};
\end{align}
therefore, for all $p_0 < p < \infty$, by the triangle inequality and the properties of the Whitney decomposition, we arrive at the bound
\begin{align}
    &\| \mathcal{V}^r\big( \sum_k \frac{1}{N} \sum_{n \leq N} \mathbf{1}_{[0,1)}(P(n) - k) T^k f : N \in I_R\big) \|_{L^p(X)} \\
    & \qquad \equiv \| \mathcal{V}^r\big(  \frac{1}{N} \sum_{n \leq N} T^{\lfloor P(n) \rfloor} f : N \in I_R \big) \|_{L^p(X)}
    \lesssim_{p,R} \| f\|_{L^p(X)}
\end{align}
for the same $r = r(p)$; in particular, this implies pointwise convergence for 
\begin{align}
    \frac{1}{N} \sum_{n \leq N} T^{\lfloor P(n) \rfloor} f 
\end{align}
along each sequence $I_R$, which implies Theorem \ref{t:ptwise0} by standard arguments, see e.g.\ \cite[Lemma 4.17]{BOOK}. One further reduction we will make is that we will truncate our sums to their ``upper-halves:" by convexity, arguing as in \cite{KMT}, we can truncate 
\[ \frac{1}{N} \sum_{n \leq N} \phi_J(P(n) - k) T^k f \longrightarrow \frac{2}{N} \sum_{N/2 < n \leq N} \phi_J(P(n) - k) T^k f, \]
so we will make this replacement without comment in what follows. 

By Calder\'{o}n's Transference principle \cite{C1}, it suffices to establish the following key proposition.

\begin{proposition}\label{p:l2}
For each $R \geq 1$, $r > 2$, $p_0 < p < \infty$, and each $J$, the following estimates holds:
    \begin{align}
       & \| \mathcal{V}^r\big( \frac{2}{N} \sum_{N/2< n \leq N} \sum_k \phi_J(P(n) - k) f(x-k) : N \in I_R\big) \|_{\ell^p(\mathbb{Z})} \lesssim_{p,r,R}  \| f \|_{\ell^p(\mathbb{Z})}.  
    \end{align}
\end{proposition}

Henceforth, for notational ease, we will suppress the subscript
\[ \phi_J \longrightarrow \phi.\]
We begin with a preliminary reduction, as per \cite[\S 8]{B2}, see also \cite[ \S 7]{BK2}; 
namely, we apply Poisson summation to express
\begin{align}
   & \sum_k \frac{2}{N} \sum_{N/2< n \leq N} \phi(P(n) - k) f(x-k) \\
   & \qquad = \sum_{\xi \in \mathbb{Z}} \int \big( \hat{f}(\beta) \hat{\phi}(\xi - \beta) \big) \cdot \big( \frac{2}{N} \sum_{N/2 < n \leq N} e((\xi-\beta) \cdot P(n) ) \big) \cdot e(\beta x) \ d\beta  \\
    & \qquad \qquad =: \sum_{\xi \in \mathbb{Z}} e(\xi b_0) \int \widehat{f_\xi}(\beta) \cdot  \big( \frac{2}{N} \sum_{N/2 < n \leq N} e((\xi - \beta) Q(n)) \big) \cdot e(\beta x) \ d\beta.
\end{align}
Since we may bound
\begin{align}
    \| f_\xi \|_{\ell^p(\mathbb{Z})} &\leq \| \hat{\phi}(\xi - \cdot) \|_{A(\mathbb{T})} \| f \|_{\ell^p(\mathbb{Z})} \lesssim |J| ( 1+ |J| |\xi|)^{-2} \| f \|_{\ell^p(\mathbb{Z})},
\end{align}
it suffices to freeze a single value of $\xi$; to establish this bound we use the regularity of $\phi$, see \eqref{e:deriv}, to estimate
\begin{align}
    \| \hat{\phi}(\xi - \cdot) \|_{A(\mathbb{T})} &\lesssim \| \hat{\phi}(\xi - \cdot) \|_{L^2(\mathbb{T})} + \| \hat{\phi}(\xi - \cdot) \|_{L^2(\mathbb{T})}^{1/2} \| \partial \hat{\phi}(\xi - \cdot) \|_{L^2(\mathbb{T})}^{1/2} \\
    & \qquad \lesssim |J| ( 1 + |J| |\xi|)^{-2}.
\end{align}

\subsection{The Proof of Proposition \ref{p:l2}}
We first reduce our task to establishing the following Proposition.
\begin{proposition}[First Reduction]\label{p:l2'}
   For each $R \geq 1, r > 2, p_0 < p < \infty$ and each $\xi \in \mathbb{Z}$, the following estimate holds uniformly:
    \begin{align}
        & \| \mathcal{V}^r \big(  \int \widehat{f}(\beta) \cdot m_N(\beta)\cdot e(\beta x) : N \in I_R \big) \|_{\ell^p(\mathbb{Z})} \lesssim_{p,r,R}  \| f \|_{\ell^p(\mathbb{Z})},  
    \end{align}
    where
    \[ m_N(\beta) = m_{N,\xi}(\beta) := \frac{2}{N} \sum_{N/2 < n \leq N} e((\xi - \beta) Q(n)). \]
\end{proposition}

We now apply the Fourier transform method, similar to the decomposition of \cite[\S 8]{B2}, to express
\begin{align}
    m_N(\beta) = L_N(\beta) +  \mathcal{E}_N(\beta) = \sum_{2^s \leq N^{\delta_0}} L_{N,s}(\beta) + \mathcal{E}_N(\beta), \; \; \; 0 < \delta_0 \leq 1/200
\end{align}
where $\|\mathcal{E}_N\|_{L^{\infty}(\mathbb{T})} \leq N^{-\delta'}$ uniformly in $\xi$, are error terms, and
\begin{align}\label{e:LNS}
L_{N,s}(\beta) := \sum_{\theta \in \mathcal{R}_{s,\xi}} S(\theta) w_N(\beta - \theta) \chi(2^{10sd}|b_d|(\beta - \theta)) \end{align}
is described below; for the details of this reduction see \cite[\S 5]{B2} or \cite[\S 4]{BOOK}.

We define the terms in order: first, $\mathbf{1}_{[-1/4,1/4]} \leq \chi \leq \mathbf{1}_{[-1/2,1/2]}$ is a smooth bump function, and $w_N$ is an analytic multiplier
\begin{align}
    w_N(\beta) :=  \widehat{\chi_d}(b_d N^d \beta) \chi(N^{d-\epsilon} \beta ), \; \; \; \chi_d(s) := \frac{2}{d s^{1-1/d}} \mathbf{1}_{(2^{-d},1)}(s),
\end{align}    
    which satisfies the bounds
\begin{align}\label{e:w} w_N(\beta) = \begin{cases} 1 + O(|b_d| N^d |\beta|) & \text{ if } |\beta| \lesssim |b_d|^{-1} N^{-d} \text{ is small} \\
O\big((1 + N^d |b_d| |\beta|)^{-1}\big) \cdot \mathbf{1}_{|\beta| \leq 10 N^{\epsilon - d}} & \text{ otherwise, } \end{cases} 
\end{align}
by standard van der Corput estimates, see e.g.\ \cite[Appendix B]{BOOK}. As for the number theoretic input, the sum runs over a subset of $\theta \in [0,1]$ satisfying 
\begin{align}\label{e:theta} \{ \theta \in [0,1] : b_j(\xi - \theta) \equiv a_j/q \mod 1, \ (a_1,\dots,a_d,q) = 1 \}, \end{align}
and we define the associated Gauss sum
\begin{align}
    S(\theta) := \frac{1}{q} \sum_{r \leq q} e(-a_1/q r - \dots - a_d/q r^d)
\end{align}
where $\theta, a_1,\dots,a_d,q$ are related as in \eqref{e:theta}. Note that the only way that \eqref{e:theta} is ever satisfied is if the coefficients of $Q$ are rationally dependent, namely
\[ Q(n) = \lambda \sum_{j=1}^d u_j n^j, \; \; \; \lambda \notin \mathbb{Q},\ u_1,\dots,u_d \in \mathbb{Z}, \ (u_1,\dots,u_d) = 1.\]
In particular, in this case we have
\[ \frac{a_j}{q} = \frac{u_j a_1}{q u_1},\]
so that whenever $\theta, \ (a_1,\dots,a_d,q)$ are related as in \eqref{e:theta}, we may express
\begin{align}
S(\theta) = \frac{1}{q} \sum_{r \leq q} e( - \frac{a_1}{q} ( r + \frac{u_2}{u_1} r^2 + \dots + \frac{u_d}{u_1} r^d)).
\end{align}

The set of frequencies $\mathcal{R}_{s,\xi} \subset \mathbb{T}$ depends on $\xi \in \mathbb{Z}$, and is defined below
\begin{align}\label{e:theta0} \{ \theta \in [0,1] : |S(\theta)| \approx 2^{-s} \};
\end{align} while these frequencies are not rational, they are fairly widely separated in that 
\[ \min_{\theta \neq \theta' \in \mathcal{R}_{s,\xi}} |\theta - \theta'| \gtrsim 5^{-ds} \]
with implicit constants depending on $Q$, as each $q$ associated to $\theta$ as in \eqref{e:theta} satisfies
\[ q \leq 2^{(d+\epsilon)s} \]
by \cite{HUA};
and, we may bound
\[ \sup_{\xi} |\mathcal{R}_{s,\xi}| \lesssim_\epsilon \min\{ 2^{(2d+\epsilon)s}, 2^{(4+\epsilon)s} \}\]
by the following lemma.
\begin{lemma}\label{l:gauss}
    Let $P(n) = \sum_{j\leq d} B_j n^j \in \mathbb{Z}{[\cdot]}$ be a polynomial with $(B_1,\dots,B_d) = 1$, and define
    \[ S(a/q) := \frac{1}{q} \sum_{r \leq q} e(-a/q P(r)).\]
Then
\[ |\{ a/q : (a,q) = 1, |S(a/q)| \approx 2^{-s} \} |\lesssim_{\epsilon,P} \min\{ 2^{(2d+\epsilon) s}, 2^{(4+\epsilon) s} \}.\]
\end{lemma}
\begin{proof}
    The case of $d \leq 2$ is Hua's Bound \cite{HUA}. The interesting case is where $d \geq 3$. Our departure point is \cite[Lemma 12.2-12.3]{IK}, which states that whenever $e \geq 2$,
\begin{align}
        \frac{1}{p^e} \sum_{r \leq p^e} e(-\frac{a}{p^e} P(r)) \lesssim_P p^{-\lfloor e/2 \rfloor}, \; \; \; (a,p) = 1,
    \end{align}
with a generic upper bound on the implicit constant on the order of $d$; the case where $e=1$ sees a square root savings is due to the Weil bound. Already, taking into account that each $q$ has at most $O( \frac{\log q}{\log \log q})$ distinct prime factors, we see that
\[ \{ a/q : (a,q) = 1, |S(a/q)| \approx 2^{-s} \} \subset \{ a/q : q \lesssim 2^{(3+\epsilon)s} \}. \]
    We improve this in what follows.

For each integer, $q \leq 2^{(3+\epsilon)s}$, we decompose
\[ q = q_1 \cdot q_2, \; \; \; (q_1,q_2) = 1\]
    where $q_2$ is cube-free, and define
    \[ t(q_1) := \prod_{p^e || q_1} p^{\lfloor e/2 \rfloor} \]
so that the map
\[ q_1 \mapsto t(q_1) \]
is at most $2^{\epsilon s}$-to-one whenever $q_1 \leq 2^{(3+\epsilon)s}$. And, we note that $t$ is a multiplicative function. If we set
\[ \alpha(q) := \max_{(a,q) = 1} |S(a/q)| \]
then, by the Chinese Remainder Theorem, we may bound
\[ \alpha(q_1 q_2) \lesssim_P t(q_1)^{\epsilon-1} q_2^{\epsilon-1/2}.  \]

Consequently, we estimate
\begin{align}
    &|\{ a/q : (a,q) = 1, \ |S(a/q)| \geq 2^{-s} \}| \leq \sum_{q_1 \leq 2^{(3+\epsilon)s}} \phi(q_1) \sum_{q_2 : \alpha(q_2) \geq t(q_1)/2^s} \phi(q_2) \\
        & \qquad \lesssim_\epsilon \sum_{q_1 \leq 2^{(3+\epsilon)s}} \phi(q_1) \cdot (\frac{2^s}{t(q_1)})^{4+\epsilon} \\
        & \qquad \qquad \leq 2^{(4+\epsilon)s} \prod_p \big(1 + \frac{p^3 - p^2}{p^{(4+\epsilon)}} + \frac{p^4 - p^3}{p^{2(4+\epsilon)}} + \dots + \frac{p^k - p^{k-1}}{p^{\lfloor k/2 \rfloor(4+\epsilon)}} + \dots \big) \\
        & \qquad \qquad \qquad \lesssim_\epsilon 2^{(4+\epsilon)s}
\end{align}
where we here set $\phi$ to be the Euler totient function, and the Euler product sees no small prime powers since all prime powers appearing in $q_1$ are $\geq 3$.   
\end{proof}

By Magyar-Stein-Wainger Transference \cite{MSW} and \cite{CRS}, we may bound
\begin{align}
\sup_N \| (L_{N,s} \hat{f})^{\vee} \|_{\ell^p(\mathbb{Z})} \lesssim 2^{-c_p s} \|  f\|_{\ell^p(\mathbb{Z})}
\end{align}
for each $p > p_0$, so for each such $p$
\begin{align}
\| \mathcal{E}_N^{\vee}* f \|_{\ell^p(\mathbb{Z})} \leq \| (m_N \hat{f})^{\vee} \|_{\ell^p(\mathbb{Z})} + \sum_{s} \| (L_{N,s} \hat{f})^{\vee} \|_{\ell^p(\mathbb{Z})} \lesssim \| f\|_{\ell^p(\mathbb{Z})}
\end{align}
while 
\[ \| \mathcal{E}_N^{\vee}*f \|_{\ell^2(\mathbb{Z})} \leq N^{-\delta'} \| f \|_{\ell^2(\mathbb{Z})}.\]
In particular
\begin{align}
\sum_{N \in I_R} \| \mathcal{E}_N f \|_{\ell^p(\mathbb{Z})} \lesssim_{p,R} \| f \|_{\ell^p(\mathbb{Z})}
\end{align}
for each $p > p_0$, so moving forward, it suffices to address the contribution from the multipliers $\{ L_{N,s} : 2^s \leq N^{\delta_0} \}$. Specifically, matters have reduced to the following proposition; for notational ease, we will henceforth restrict all times $N$ to be larger than $2^{\delta_0^{-1} s}$. Moving forward, since our multipliers are all compactly supported inside of $\mathbb{T}$, we will appeal to Magyar-Stein-Wainger transference \cite{MSW} and concern ourselves with real variable estimates. 

\begin{proposition}[Second Reduction]\label{p:almost}
For each $r > 2, p > p_0$, and $s \geq 1$, there exists $c_p > 0$ so that following estimate holds, independent of $\xi$:
\begin{align}\label{e:varL}
    \| \mathcal{V}^r\big( (L_{N,s} \hat{f})^{\vee} : N \in I_R \big) \|_{L^p(\mathbb{R})} \lesssim_{p,r,R} 2^{-c_p s} \| f\|_{L^p(\mathbb{R})}.
\end{align}
\end{proposition}
After rescaling, we are now in a position to apply Theorem \ref{t:L2}.
To adapt to the notation there, we set
\[ \widehat{f_\theta}(\beta) := S(\theta) \chi_0(2^{100sd} |b_d| \beta) \hat{f}(\beta+\theta)\]
where $\chi_0$ is like $\chi$ but is one on its support, and increase the density of the sequence of times $\Phi_k \longrightarrow \Phi_{k/R}$, which introduces only notational changes and constants that depend on $R$. The upshot is that
\begin{align}
\| \mathcal{V}^r\big( (L_{N,s} \hat{f})^{\vee} : N \in I_R \big) \|_{L^p(\mathbb{R})} &\lesssim_{p,R} (\frac{r}{r-2})^2 s^2 \| (\sum_\theta |f_\theta|^2)^{1/2} \|_{L^p(\mathbb{R})} \\
& \qquad \lesssim_{r,p,R} s^2 2^{-s} |\mathcal{R}_s|^{1/p -1/2 + \epsilon} \lesssim 2^{-c_p s} \| f \|_{L^p(\mathbb{R})},
\end{align}
whenever $p > p_0$.

\section{Appendix}
In this section we provide an alternative proof of Theorem \ref{t:main0}, due to M. Wierdl, that addresses functions in all $L^p(X)$ spaces, $1 < p  \leq \infty$, whenever $X$ is a Lesbesgue space; the argument derives from sophisticated arithmetic considerations, namely Ionescu-Wainger theory \cite{IW} or Bourgain's double logarithmic Lemma \cite{B2} used in \cite{MT}, and higher-dimensional analysis, \cite{MST}.

Let $(X,\mu,T)$ be such a measure preserving system, and let $f \in L^p(X), \ 1 < p < \infty$ be arbitrary. 

As above, we fix
\[ P(n) = \sum_j b_j n^j. \]

We first assume our transformation $T$ generates a flow,
\[ t \mapsto T^t; \]
if we then define
\[ U_j := T^{b_j},\]
so that the $\{ U_j \}$ commute, and 
\[ f(T^{P(n)} x) \equiv f(U_1^n \circ U_2^{n^2} \circ \dots \circ U_d^{n^d} x), \]
by \cite{MST}, for each $1 < p < \infty, \ r > 2$
\begin{align}
    \| \mathcal{V}^r( \frac{1}{N} \sum_{n \leq N} f(U_1^n \circ U_2^{n^2} \circ \dots \circ U_d^{n^d}) : N ) \|_{L^p(X)} \lesssim_p \frac{r}{r-2} \| f \|_{L^p(X)},
\end{align}
which implies that for each $f \in L^p(X)$
\[ \frac{1}{N} \sum_{n \leq N} f(T^{P(n)} x) \]
converges almost everywhere, so the result then follows from \cite[Theorem 1]{Bosh}, see also \cite{Les}.

It remains to reduce to the case of flows; we use suspensions. So, suppose that $(X,\mu,T)$ is an arbitrary Lesbesgue space, and consider the space
\[ Y := X \times [0,1) = \{ (x,s) : x \in X, \ s \in [0,1)\}.\]
On $Y$, equipped with the product measure $\nu := \mu\otimes dx$, define the flow
\[ \tilde{T}^t(x,s) := (T^{\lfloor s+t \rfloor} x, s+t \mod 1 )\]
Now, let $f \in L^p(X)$ for $1 < p \leq \infty$ be arbitrary but fixed; our job is to show that
\[ \frac{1}{N} \sum_{n \leq N} f(T^{\lfloor P(n) \rfloor} x) \]
converges $\mu$-almost surely. But, if we define
\[ \tilde{f}(x,s) := f(x) : Y \to \mathbb{C},\]
then we know that
\[ \frac{1}{N} \sum_{n \leq N} \tilde{f}(\tilde{T}^{\lfloor P(n) \rfloor}(x,s)) \]
converges $\nu$-almost surely, in $(x,s)$. In particular, by Fubini, for almost every $s \in [0,1)$
\[ \frac{1}{N} \sum_{n \leq N} \tilde{f}(\tilde{T}^{\lfloor P(n) \rfloor}(x,s)) \]
converges $\mu$-almost surely; suppose that $s_0$ is such an $s$, and let $X_0 \subset X$ be the corresponding full-$\mu$-measure subset so that for all $x \in X_0$
\[ \frac{1}{N} \sum_{n \leq N} \tilde{f}(\tilde{T}^{\lfloor P(n) \rfloor}(x,s_0)) \]
converges.

We observe that
\begin{align}
    \tilde{f}\big(\tilde{T}^{\lfloor P(n) \rfloor}(x,s_0)\big) &= \tilde{f}\big(T^{\lfloor s_0 + (\lfloor P(n) \rfloor) \rfloor} x, s_0 +\lfloor P(n) \rfloor \big) \\
    & \qquad = \tilde{f}\big(T^{\lfloor P(n) \rfloor} x, s_0 +\lfloor P(n) \rfloor \big) = f(T^{\lfloor P(n) \rfloor} x),
\end{align}
so for all $x \in X_0$,
\begin{align}
    \lim_N \frac{1}{N} \sum_{n \leq N} f(T^{\lfloor P(n) \rfloor}x) &= \lim_N \frac{1}{N} \sum_{n \leq N} \tilde{f}(\tilde{T}^{\lfloor P(n) \rfloor}(x,s_0))
\end{align}
exists.

\end{document}